\documentclass{amsart}

\usepackage{amscd}
\usepackage{amsmath, amssymb}
\usepackage{amsfonts}

\newcommand{\dbar}{\overline{\partial}}
\newcommand{\ddt}{\frac{\partial}{\partial t}}

\newcommand{\ov}[1]{\overline{#1}}

\newcommand{\tr}[2]{\textrm{tr}_{#1}{#2}}

\newcommand{\ve}{\varepsilon}

\renewcommand{\leq}{\leqslant}
\renewcommand{\geq}{\geqslant}
\renewcommand{\le}{\leqslant}
\renewcommand{\ge}{\geqslant}

\newcommand{\be}{\begin{equation}}
\newcommand{\ee}{\end{equation}}

\begin{document}
\newcounter{theor}
\setcounter{theor}{1}
\newtheorem{claim}{Claim}
\newtheorem{theorem}{Theorem}[section]
\newtheorem{lemma}[theorem]{Lemma}
\newtheorem{corollary}[theorem]{Corollary}
\newtheorem{proposition}[theorem]{Proposition}
\newtheorem{question}{question}[section]
\newtheorem{defn}{Definition}[theor]
\newtheorem{remark}{Remark}[section]

\newenvironment{example}[1][Example]{\addtocounter{remark}{1} \begin{trivlist}
\item[\hskip
\labelsep {\bfseries #1  \thesection.\theremark}]}{\end{trivlist}}

\title[The degenerate J-flow]{The degenerate J-flow and the Mabuchi energy on minimal surfaces of general type}%

\author{Jian Song}
\address{Jian Song, Department of Mathematics,
Rutgers}
\email{jiansong@math.rutgers.edu}
\author{Ben Weinkove}
\address{Ben Weinkove, Mathematics Department,
Northwestern University}
\email{weinkove@math.northwestern.edu}

\thanks{}%
\footnotetext[1]{Research supported in part by NSF grants  DMS-08047524 and DMS-1105373.  This work was carried out while the second-named author was on leave from UC San Diego.}

\begin{abstract}{ We prove existence, uniqueness and convergence of solutions of the degenerate J-flow on K\"ahler surfaces.  As an application, we establish the 
 properness of the Mabuchi energy for K\"ahler classes in a certain subcone of the K\"ahler cone on  minimal surfaces of general type.  }
\end{abstract}

\maketitle

\section{Introduction}

Let $M$ be a compact K\"ahler surface with two K\"ahler metrics $\omega_0$ and $\chi_0$.  Let $\mathcal{P}_{\chi_0}$ be the space of smooth functions $\varphi$ with $\chi_{\varphi}:= \chi_0 + dd^c \varphi>0$.  The \emph{J-flow} is a parabolic flow defined on $\mathcal{P}_{\chi_0}$ by
\begin{equation} \label{jflowsmooth}
\ddt{} \varphi = c_0 - \frac{2 \chi_{\varphi} \wedge \omega_0}{\chi_{\varphi}^2}, \quad \varphi|_{t=0} =\varphi_0 \in \mathcal{P}_{\chi_0},
\end{equation}
where $c_0$ is defined by
$$c_0 = \frac{2 [\chi_0] \cdot [ \omega_0] }{[\chi_0]^2}.$$
The J-flow was introduced by Donaldson \cite{D1} in the setting of moment maps and by Chen \cite{Ch1} as the gradient flow of the $\mathcal{J}$-functional which appears in his formula for the Mabuchi energy.
Smooth solutions to (\ref{jflowsmooth}) exist for all time and are unique \cite{Ch2}.  

Under the assumption 
\begin{equation} \label{dc}
c_0 [\chi_0] - [\omega_0] >0,
\end{equation}
it was shown in \cite{W1} that the solution to the J-flow converges smoothly to $\varphi_{\infty}$ solving the critical equation
\begin{equation} \label{critical1}
2 \chi_{{\varphi}_{\infty}} \wedge \omega_0 = c_0 \chi_{\varphi_{\infty}}^2.
\end{equation}
The fact that smooth solutions to (\ref{critical1}) exist under the condition (\ref{dc}) was conjectured by Donaldson \cite{D1} and proved by Chen by reducing the equation to the complex Monge-Amp\`ere equation solved by Yau \cite{Y1}.
In higher dimensions, it was shown in \cite{W2} that the flow converges under the cohomological assumption $c_0[\chi] - (n-1)[\omega_0] >0$.    Necessary and sufficient conditions for convergence of the J-flow in terms of $[\chi_0]$ and $\omega_0$ were found in \cite{SW1}. 

In \cite{FLM, FL1}, convergence results were proved for generalizations of the J-flow known as inverse $\sigma_k$-flows.  In \cite{FL2}, Fang-Lai analyzed the behavior of the inverse $\sigma_k$-flow on general K\"ahler classes for metrics with  Calabi symmetry.  The J-flow  has been investigated on Hermitian manifolds by Y. Li \cite{LiY}, and the critical equation on Hermitian manifolds with boundary  by Guan-Li \cite{GL}.  
An equation bearing strong similarities to the critical equation for the J-flow is the \emph{complex Hessian equation} (see for example \cite{LiS, B, H, HMW, Chi}); it has been studied intensely in the last few years, and the existence of solutions on compact K\"ahler manifolds was recently established by Dinew-Ko{\l}odziej \cite{DK}.
  
We now return to the discussion of the J-flow. 
In complex dimension two, the behavior of the J-flow was investigated \cite{FLSW} in the case 
where $c_0 [\chi_0]-[\omega_0] \ge 0$, where $\beta \ge 0$ means that the cohomology class $\beta$ admits a smooth nonnegative representative.  A uniform $L^{\infty}$ estimate for $\varphi$ was established, and it was shown that the J-flow converges smoothly to a singular K\"ahler metric away from a finite number of curves of negative self-intersection.

In this paper, we generalize the result of \cite{W1} in a different direction.  We consider the case where $\omega_0$ is no longer a K\"ahler metric, but a closed  (1,1) form satisfying a certain nonnegativity condition.  More precisely, assume that we have a background K\"ahler metric $\hat{\omega}$ and an effective divisor $D$ on $M$ with associated line bundle $[D]$. Let $H$ be a fixed Hermitian metric on the line bundle $[D]$, and let $s$ be a holomorphic section of $[D]$ which vanishes exactly along $D$.  Our assumption on $\omega_0$ is:
\begin{equation} \label{omega0}
\omega_0 \  \textrm{is a smooth closed (1,1) form with } \omega_0 \ge \frac{1}{C_0} |s|^{2\beta}_H \hat{\omega} \ \textrm{and} \ \omega_0 - \rho R_H \ge \frac{1}{C_0} \hat{\omega},
\end{equation}
for some positive constants $C_0, \beta, \rho$.  Here, $R_H  = - dd^c \log H$ denotes the curvature form of the Hermitian metric $H$. 

Clearly (\ref{omega0}) holds  if $\omega_0$ is K\"ahler.  It is not uncommon for non-K\"ahler cohomology classes to admit a closed form $\omega_0$ satisfying (\ref{omega0}).    Indeed, we will see below that if $M$ is a minimal surface of general type then the canonical class, if it is not ample, is such an example.   Also, it is well-known that such classes can be found on the boundary of the K\"ahler cone of blow-ups of K\"ahler manifolds, as discussed in \cite{SW2} for example.

 We call the equation (\ref{jflowsmooth}) with $\omega_0$ satisfying (\ref{omega0}) the 
 \emph{degenerate J-flow}.  Now (\ref{jflowsmooth}) is no longer a parabolic equation in general and we cannot expect to obtain smooth solutions.  The main result of this paper is that there exists a unique ``weak'' solution to this degenerate parabolic equation, which converges to a ``weak'' solution of the critical equation.  More precisely, define a space 
 \begin{equation}
 \mathcal{P}^{\textrm{weak}}_{\chi_0}:= \{ \varphi \in C^{\infty}(M \setminus D) \cap \textrm{PSH}_{\chi_0}(M) \cap L^{\infty}(M) \ | \ \chi_0  + dd^c \varphi >0 \textrm{ on } M\setminus D \}.
 \end{equation}
 where we are writing $D$ also for the subset of $M$ defined by the divisor $D$.  Here $\textrm{PSH}_{\chi_0}(M)$ consists of upper semicontinuous functions $\varphi: M \rightarrow [-\infty, \infty)$ such that $\varphi+\psi_0$ is plurisubharmonic, where $\psi_0$ is a (smooth) local K\"ahler potential for $\chi_0$.

We have the following result.

\begin{theorem} \label{maintheorem}  Let $M$ be a compact K\"ahler surface, with K\"ahler metrics $\chi_0$ and $\hat{\omega}$.  Assume that $\omega_0$ satisfies (\ref{omega0}) and assume that
\begin{equation} \label{condition}
c_0 [\chi_0] - [\omega_0] >0, \quad \textrm{for} \quad c_0 = \frac{2 [\chi_0] \cdot [ \omega_0] }{[\chi_0]^2}.
\end{equation}
For any smooth $\varphi_0 \in \mathcal{P}_{\chi_0}$, 
 there exists a unique solution $\varphi=\varphi(t) \in \mathcal{P}^{\emph{weak}}_{\chi_0}$ of the degenerate J-flow
\begin{equation} \label{jflowdegen}
\ddt{} \varphi = c_0 - \frac{2 \chi_{\varphi} \wedge \omega_0}{\chi_{\varphi}^2} \textrm{ on } M \setminus D, \quad \varphi|_{t=0} =\varphi_0,
\end{equation}
with $\displaystyle{\sup_{M \setminus D} \left| \frac{\partial{\varphi}}{\partial t} \right|}$ bounded (independent of $t$).  

As $t\rightarrow \infty$, 
$$\varphi(t) \overset{C^{\infty}_{\emph{loc}}(M\setminus D)}{\longrightarrow} \varphi_{\infty},$$
where $\varphi_{\infty} \in \mathcal{P}^{\emph{weak}}_{\chi_0}$ solves the critical equation
\begin{equation} \label{criteq}
2 \chi_{{\varphi}_{\infty}} \wedge \omega_0 = c_0 \chi_{\varphi_{\infty}}^2 \quad \textrm{on} \quad M \setminus D.
\end{equation}
Moreover, $\varphi_{\infty} \in \mathcal{P}^{\emph{weak}}_{\chi_0}$ is the unique solution of the critical equation up to the addition of a constant.
\end{theorem}

Note  that $\varphi_{\infty}$ coincides with the pluripotential solution of (\ref{criteq}) on $M$ in the sense of Bedford-Taylor (see \cite{K3}, for example).

We recall now the $\mathcal{J}$-functional of Chen \cite{Ch1}.  
Given a K\"ahler form $\chi_0$ and a closed $(1,1)$ form $\omega_0$, define the $\mathcal{J}$-functional by
\begin{equation} \label{Jfunctional}
\mathcal{J}_{\omega_0, \chi_0}(\varphi) = \int_0^1 \int_M \dot{\varphi}_s (2\chi_{\varphi_s} \wedge \omega_0 - c_0 \chi_{\varphi_s}^2 ) ds, \quad \textrm{for } \varphi \in \mathcal{P}_{\chi_0},
\end{equation}
where $\varphi_s$ is a smooth path in $\mathcal{P}_{\chi_0}$ between 0 and $\varphi$.  In the case where $\omega_0$ is K\"ahler, the J-flow is the gradient flow of the $\mathcal{J}$-functional.   A consequence of our main result is that the $\mathcal{J}$-functional is uniformly bounded from below on the space $\mathcal{P}_{\chi_0}$.


\begin{corollary} \label{corollarylb}
As in Theorem \ref{maintheorem}, let $M$ be a compact K\"ahler surface with K\"ahler metrics $\chi_0$ and $\hat{\omega}$.  Assume that $\omega_0$ satisfies (\ref{omega0}) and that (\ref{condition}) holds.  Then there exists a constant $K$ depending only on the fixed data $M, \omega_0, \chi_0$ such that 
\begin{equation}
\mathcal{J}_{\omega_0, \chi_0} (\varphi) \ge K, \quad \textrm{for all } \varphi \in \mathcal{P}_{\chi_0}.
\end{equation}
\end{corollary}

This result has an immediate application to the \emph{Mabuchi energy functional}, which we now explain.
A well-known open problem in K\"ahler geometry is to determine which K\"ahler classes on $M$ admit K\"ahler metrics of constant scalar curvature.  The Yau-Tian-Donaldson conjecture relates this to a notion of stability in the sense of geometric invariant theory \cite{Y2,T1,D2}.
 A related question is to ask instead for  which K\"ahler classes is the Mabuchi energy functional  \emph{proper} (see Section \ref{sectionproofs} below for the definition).  Indeed, according to a conjecture of Tian \cite{T2}, these questions are essentially equivalent, modulo some issues which arise if $M$ admits holomorphic vector fields.

It was shown by Chen \cite{Ch1} that if the canonical bundle $K_M$ of $M$ satisfies $K_M>0$ then the Mabuchi energy is bounded below on all K\"ahler classes $[\chi_0]$ satisfying
\begin{equation} \label{assume1}
\left( \frac{2 [\chi_0] \cdot K_M}{[\chi_0]^2} \right) [\chi_0] - K_M >0.
\end{equation}
An alternative proof of this was given by the second-named author using the J-flow \cite{W1}.  Later, the authors observed \cite{SW1} that under the same assumption  (\ref{assume1}), it follows from a result of Tian \cite{T2} that in fact the Mabuchi energy is not just bounded below but proper. 
 Moreover, we proved analogous results  on manifolds $M$ of any dimension with $K_M>0$.   
 
  Fang-Lai-Song-Weinkove \cite{FLSW} recently showed that the assumption (\ref{assume1}) can be weakened to
\begin{equation} \label{assumeFLSW}
\left( \frac{2 [\chi_0] \cdot K_M}{[\chi_0]^2} \right) [\chi_0] - K_M  \ge 0,
\end{equation}
where $\sigma \ge 0$ means that the cohomology class $\sigma$ admits a smooth nonnegative representative.  

In this paper, we instead allow $K_M$ to satisfy a more general nonnegativity condition than being ample.  We assume that $K_M$ is big and nef, meaning that $K_M^2 >0$ and $K_M\cdot C \ge 0$ for all curves $C$ on $M$.
 This is equivalent to saying that $M$ is a minimal surface of general type.

\begin{corollary} \label{mabuchithm}
Let $M$ be a minimal surface of general type.  Then the Mabuchi energy is proper on all K\"ahler classes  $[\chi_0]$ satisfying
\begin{equation} \label{assume2}
\left( \frac{2 [\chi_0] \cdot K_M}{[\chi_0]^2} \right) [\chi_0] - K_M >0.
\end{equation}
\end{corollary}

Thus one would expect constant scalar curvature K\"ahler metrics to exist in these classes.  It is also expected that such classes are \emph{K-stable} in the sense of \cite{T1, D2}.  In fact, it follows immediately from results of Panov-Ross (see the argument of \cite[Example 5.9]{PR})  that algebraic classes satisfying (\ref{assume2}) are \emph{slope stable} in the sense of Ross-Thomas  \cite{RT}.

The main technical results are contained in Section \ref{sectionjflow}.  The idea is to replace the degenerate $(1,1)$-form $\omega_0$ with a K\"ahler form $\omega_{\ve}$ for $\ve>0$ and obtain estimates for the J-flow away from $D$ which are independent of $\ve$.  As $\ve \rightarrow 0$, we have $\omega_{\ve} \rightarrow \omega_0$ and we obtain a solution of the degenerate J-flow (cf. results of Song-Tian \cite{ST} in the case of the K\"ahler-Ricci flow).  The  key estimates are contained in Proposition \ref{mainprop}.  In Section \ref{sectionproofs}, we prove Theorem \ref{maintheorem} and its corollaries.

Finally, some words about notation.  When we are given a (1,1) form $\beta$, we will often define a tensor with components $\beta_{i\ov{j}}$ by $\beta = \sqrt{-1} \beta_{i\ov{j}}dz^i \wedge d\ov{z^j}$.  An exception to this notation is that for a  K\"ahler form $\omega$ we write $\omega = \sqrt{-1} g_{i\ov{j}}dz^i \wedge d\ov{z^j}$, and similarly for $\omega_0, g_0$, etc.  Given a positive definite $(1,1)$ form $\alpha$ and a $(1,1)$ form $\beta$, we write $\tr{\alpha}{\beta}$ for $\alpha^{\ov{j}i} \beta_{i\ov{j}}$, where $(\alpha^{\ov{j}i})$ is the inverse of $(\alpha_{i\ov{j}})$.  We will often denote uniform constants by $C, C_0, C_1, C', C'', \ldots$ etc., which may differ from line to line.

The authors thank the referee for some helpful comments and suggestions.

\section{Estimates for solutions of the J-flow} \label{sectionjflow}

\setcounter{equation}{0}

The degenerate J-flow (\ref{jflowdegen}) defined in the introduction is not a parabolic equation. We peturb the equation to make it parabolic.  

Assume we are in the setting of Theorem \ref{maintheorem}.
 Write $\omega_{\ve} = \omega_0 + \ve \hat{\omega} >0$ and define
\begin{equation} \label{ce}
c_{\ve} := 2 \frac{[\chi_0] \cdot [\omega_{\ve}]}{[\chi_0]^2}. 
\end{equation}
As $\ve \rightarrow 0$,  $c_\ve \rightarrow c_0$. From (\ref{condition}), we may choose $\ve_0>0$ sufficiently small so that for $\ve \in [0, \ve_0]$ we have
\begin{equation} \label{cohomve}
c_{\ve}  [\chi_0] - [\omega_{\ve}]>0.
\end{equation}
Then consider the family of J-flows
\begin{equation} \label{Jflowve}
\ddt{} \varphi_{\ve} = c_{\ve} - \frac{2\chi_{\varphi} \wedge \omega_{\ve}}{\chi_{\varphi}^2},
 \qquad \varphi_{\ve}|_{t=0} = \varphi_0,
\end{equation}
parametrized by $\ve \in (0, \ve_0]$.
By Chen's result \cite{Ch2}, we know that there exists a solution to  (\ref{Jflowve}) on $M \times [0,\infty)$.  The main result of this section is that we have  uniform (independent of $\ve$) $L^{\infty}$ bounds for $\varphi_{\ve}$, $\dot{\varphi}_{\ve}$ and $C^{\infty}$ estimates for $\varphi_{\ve}$ away  from $D$.   

\begin{proposition} \label{mainprop}
In the setting described above,  there exists a uniform constant $C$ such that for all $t$ and $\ve \in (0, \ve_0]$,
\begin{equation}
\| \varphi_{\ve} \|_{L^{\infty}(M )} \le C \quad \textrm{and} \quad \| \dot{\varphi}_{\ve} \|_{L^{\infty}(M )} \le C.
\end{equation}
Moreover, for any compact set $\Omega$ of $M\setminus D$ and $k>0$, there exists $C_{\Omega, k}>0$ such that  
\begin{equation}\label{ckest}
||\varphi_\ve||_{C^k(\Omega, \chi_0)} \leq C_{\Omega, k}.
\end{equation}

\end{proposition}

In order to prove the proposition, it suffices to establish the uniform $L^\infty$ estimates on $M$ and a second order estimate on $M\setminus D$. 

\begin{lemma} There exists a uniform constant $C$ such that for all $t$ and $\ve \in (0, \ve_0]$,
\begin{equation}\label{c_0est}
\| \varphi_{\ve} \|_{L^{\infty}(M )} \le C.
\end{equation}

\end{lemma}

\begin{proof}
Put 
\begin{equation} \label{alpha}
\alpha_{\ve} := c_{\ve} \chi_0 - \omega_{\ve} >0.
\end{equation}
We are allowed to assume without loss of generality that $\alpha_{\ve}$ is K\"ahler and in addition, there exists $\delta>0$ such that $\alpha_\ve > \delta \chi_0$ for all $\ve \in [0, \ve_0]$. This is possible since the condition (\ref{condition}) implies that we can find a smooth function $\eta$ with $c_0 \chi_0 - \omega_0 + dd^c \eta > 2\delta \chi_0$ for a small $\delta>0$.  Shrinking $\ve_0$ if necessary, we may assume that $c_{\ve} \chi_0 -\omega_{\ve} + dd^c \eta > \delta \chi_0$ for all $\ve \in [0,\ve_0]$, and we can estimate $\varphi_\ve- \eta$ instead of $\varphi_\ve$. 

We begin by proving an $L^{\infty}$ estimate for $\varphi_{\ve}$ which is independent of $\ve$.  We follow an argument similar to that in \cite{FLSW}.  It uses the trick of \cite{Ch1} of rewriting the critical equation as a complex Monge-Amp\`ere equation, together with Yau's $L^{\infty}$ estimate.

There exists a smooth solution $\psi_{\ve}$  of the equation
\be \label{yau}
(\alpha_{\ve} + c_{\ve} dd^c \psi_{\ve})^2 = \omega_{\ve}^2, \quad \alpha_{\ve} + c_{\ve} dd^c \psi_{\ve} >0, \quad \sup_M \psi_{\ve} =0.
\ee
Indeed, since $$[\alpha_{\ve}]^2 = c^2_{\ve} [\chi_0]^2 + [\omega_{\ve}]^2 - 2c_{\ve} [\chi_0]\cdot[\omega_{\ve}] = [\omega_{\ve}]^2,$$
this follows from Yau's theorem \cite{Y1}.  Moreover, 
$$\| \psi_{\ve} \|_{L^{\infty}} \le C,$$
for $C$  independent of $\ve$. This follows from Yau's original proof using Moser's iteration, since $\alpha_\ve$ is a smooth family of K\"ahler metrics which satisfy
$\alpha_{\ve} > \delta \chi_0$ and hence $\omega_\ve^2/ \alpha_\ve^2$ is uniformly bounded from above.  Here we are using the fact that the Sobolev constant, used in Yau's iteration argument,  remains bounded on any set of Riemannian metrics which is compact in the $C^0$ topology and has a uniform lower bound.

  Note that Yau's $C^{\infty}$ estimates for $\psi_{\ve}$ may indeed depend on $\ve$, but in what follows we only need uniformity in the $L^{\infty}$ estimate.  Observe that 
$$\chi_{\psi_{\ve}} = \chi_0 + dd^c \psi_{\ve} >0,$$
since $\chi_{\psi_{\ve}} =\frac{1}{c_{\ve}} (\alpha_{\ve}+ c_{\ve} dd^c \psi_{\ve} + \omega_{\ve})$ and $\alpha_{\ve} + c_{\ve} dd^c \psi_{\ve} >0$.

We have 
$$\omega_{\ve}^2 = (c_{\ve} \chi_{\psi_{\ve}} -\omega_{\ve})^2 = c_{\ve}^2 \chi_{\psi_{\ve}}^2 - 2 c_{\ve} \chi_{\psi_{\ve}} \wedge \omega_{\ve} + \omega_{\ve}^2,$$ and hence $\chi_{\psi_{\ve}}$ satisfies the critical equation
$c_{\ve} \chi_{\psi_{\ve}}^2 = 2\chi_{\psi_{\ve}} \wedge \omega_{\ve}$.

Now put $\theta_{\ve} = \varphi_{\ve} - \psi_{\ve}$ and compute
\begin{align} 
\ddt{\theta_{\ve}} & = \ddt{\varphi_{\ve}}
 = \frac{2\chi_{\psi_{\ve}} \wedge \omega_{\ve}}{\chi_{\psi_{\ve}}^2} - \frac{2\chi_{\varphi_{\ve}} \wedge \omega_{\ve}}{\chi_{\varphi_{\ve}}^2} 
 = \int_0^1 \frac{d}{dv} \left( \frac{2 \eta_v \wedge \omega_{\ve}}{\eta_v^2} \right) dv,
\end{align}
where $\eta_v = v\chi_{\psi_{\ve}} + (1-v) \chi_{\varphi_{\ve}}$ for $v\in [0,1]$. 
 Define $\tau_v^{\ov{\ell}k} = \eta_v^{\ov{j}k} \eta_v^{\ov{\ell} i} (g_{\ve})_{i \ov{j}}$, which is positive definite, and compute \begin{align} \nonumber
\ddt{\theta_{\ve}} & = \int_0^1 \left( \frac{d}{dv} \eta_v^{\ov{j}i} (g_{\ve})_{i \ov{j}} \right) dv \\ \nonumber 
& = - \int_0^1 \eta_v^{\ov{j} k} \eta_v^{\ov{\ell}i} \left( \frac{d}{dv} \eta_v \right)_{k \ov{\ell}} (g_{\ve})_{i \ov{j}}dv \\ \nonumber
& = - \int_0^1 \tau_v^{\ov{\ell} k} (\chi_{\psi_{\ve}} - \chi_{\varphi_{\ve}})_{k \ov{\ell}} dv \\ \label{thetaevolve}
& =  \left( \int_0^1 \tau_v^{\ov{\ell}k} dv \right) \partial_k \partial_{\ov{\ell}} \theta_{\ve}.
\end{align}
Since $ \left( \int_0^1 \tau_v^{\ov{\ell}k} dv \right)$ is a positive definite tensor, we apply the maximum principle to see that $\theta_{\ve}$ is uniformly bounded  by $\sup_M |\theta_{\ve}|$ at $t=0$.  Since $\varphi_{\ve}|_{t=0} = \varphi_0$ and $\psi_{\ve}$ is uniformly bounded it follows that $\theta_{\ve}$ is uniformly bounded independent of $\ve$.   Hence $\varphi_{\ve}$ is uniformly bounded independent of $\ve$.
\end{proof}

Next we estimate the time derivative of $\varphi_{\ve}$.  First some notation: define an operator $\tilde{\Delta}_{\ve} := h_{\ve}^{\ov{j} i} \partial_i \partial_{\ov{j}},$
where $h_{\ve}^{\ov{j} i} := \chi_{\varphi_{\ve}}^{\ov{j} p} \chi_{\varphi_{\ve}}^{\ov{q} i} (g_{\ve})_{p \ov{q}}$.   Then:

\begin{lemma} \label{lemmalbchi}
There exists a uniform constant $C$ such that for all $t$ and $\ve \in (0, \ve_0]$,
$$\| \dot{\varphi}_{\ve} \|_{L^{\infty}(M)} \le C.$$
Hence we have
\begin{equation} \label{lbchi}
\chi_{\varphi_{\ve}} \ge \frac{1}{C'} \omega_{\ve} = \frac{1}{C'}(\omega_0 + \ve \hat{\omega}),
\end{equation}
for a uniform $C'>0$.
\end{lemma}
\begin{proof} This follows immediately from the maximum principle as in \cite{Ch2}.  Indeed,
differentiating (\ref{Jflowve}) we obtain
$$\ddt{} \dot{\varphi}_{\ve} = \tilde{\Delta}_{\ve} \dot{\varphi}_{\ve}.$$
Then by the maximum principle, $\dot{\varphi}_{\ve}$ is bounded uniformly in time.  Moreover, the bound is independent of $\ve$.  In particular,
$\tr{\chi_{\varphi_{\ve}}}{\omega_{\ve}} \le C',$ and this gives (\ref{lbchi}).
\end{proof}

It is important to note Lemma \ref{lemmalbchi} does not give a uniform bound for $\chi_{\varphi_{\ve}}$ away from zero which is independent of $\ve$.  In particular, we have no \emph{a priori} upper bounds for $\tr{\chi_{\varphi_{\ve}}}{\hat{\omega}}$ or $\tr{h_{\ve}}{\hat{g}}:= h_{\ve}^{\ov{j}i} (\hat{g})_{i\ov{j}}$.

Next we wish to prove an estimate for $\chi_{\varphi_{\ve}}$.  For ease of notation, we drop all subscripts $\ve$ and write $\chi$ for $\chi_{\varphi}$.
Write $u = \tr{\hat{\omega}}{\chi}$.  We denote by $\hat{R}_{k\ov{\ell}i \ov{j}}$ the curvature of $\hat{g}$, and raise indices using $\hat{g}$. 

We first derive an evolution equation and differential inequality for $\log u$.

\begin{lemma} \label{lemmaevolve} The evolution equation of $\log u$ is given by
\begin{align} \nonumber
\left( \ddt{} - \tilde{\Delta} \right) \log u   = {} &  \frac{1}{u} \left(- h^{\ov{\ell} k } \hat{R}_{k \ov{\ell}}^{\ \ \ \ov{j} i} \chi_{i \ov{j}} - \hat{g}^{\ov{j}i} h^{\ov{s}p} \chi^{\ov{q} r} \hat{\nabla}_i \chi_{r \ov{s}}\hat{\nabla}_{\ov{j}} \chi_{p \ov{q}}  \right. \\ \nonumber
&    - \hat{g}^{\ov{j}i} h^{\ov{q}r} \chi^{\ov{s} p} \hat{\nabla}_i \chi_{r\ov{s}} \hat{\nabla}_{\ov{j}} \chi_{p\ov{q}}  
 + 2\emph{Re}\left( \hat{g}^{\ov{j}i} \chi^{\ov{q} k} \chi^{\ov{\ell} p} \hat{\nabla}_{\ov{j}} \chi_{p\ov{q}} \hat{\nabla}_i g_{k\ov{\ell}} \right) \\
 &  \left.  - \hat{g}^{\ov{j}i} \chi^{\ov{\ell} k} \hat{\nabla}_i \hat{\nabla}_{\ov{j}} g_{k \ov{\ell}} +  \chi^{\ov{\ell}k} \hat{R}_{\ \, \ov{\ell}}^{\ov{q}} g_{k \ov{q}}  + \frac{ | \partial u|^2_h}{u} \right).  \label{evolve1}
\end{align}
Moreover, there exists a constant $C$ depending only on $\hat{g}$ and $\| g_0\|_{C^2(M, \hat{g})}$ such that
\begin{align} \nonumber
\left( \ddt{} - \tilde{\Delta} \right) \log u  & \le C \emph{tr}_{h}{\hat{g}} + \frac{C}{u} (\emph{tr}_{\chi}{\hat{\omega}})(\emph{tr}_{\omega}{\hat{\omega}})  + 2 \emph{Re} \left( \chi^{\ov{s} k}  \left( \frac{\partial_k u}{u^2} \right)  \partial_{\ov{s}} \emph{tr}_{\hat{\omega}} \omega \right).
\end{align}
\end{lemma}
\begin{proof}
For any fixed $p \in M$, we choose a holomorphic coordinate system centered at $p$ with the property that $(\partial_k \hat{g}_{i\ov{j}})|_p =0$ for all $i,j,k$.  Compute at $p$,
\begin{align*}
\ddt{} \tr{\hat{\omega}}{\chi}  = {} &\ddt{} \left( \hat{g}^{\ov{j}i} \partial_i \partial_{\ov{j}} \varphi \right)  = -\hat{g}^{\ov{j}i} \partial_i \partial_{\ov{j}} (\chi^{\ov{\ell}k} g_{k \ov{\ell}}) \\
   = {} &- \hat{g}^{\ov{j}i} \partial_i ( - \chi^{\ov{q}k} \chi^{\ov{\ell}p} (\partial_{\ov{j}} \chi_{p\ov{q}} )g_{k\ov{\ell}} + \chi^{\ov{\ell} k} \partial_{\ov{j}} g_{k \ov{\ell}} ) \\
   = {} & \hat{g}^{\ov{j}i} \{ g_{k \ov{\ell}} \chi^{\ov{q}k} \chi^{\ov{\ell}p} \partial_i \partial_{\ov{j}} \chi_{p \ov{q}} - \chi^{\ov{s}k} \chi^{\ov{q}r} (\partial_i \chi_{r\ov{s}}) \chi^{\ov{\ell}p} (\partial_{\ov{j}} \chi_{p\ov{q}}) g_{k \ov{\ell}}  \\
  {} & - \chi^{\ov{q}k} \chi^{\ov{s}p} \chi^{\ov{\ell}r} (\partial_i \chi_{r\ov{s}}) (\partial_{\ov{j}} \chi_{p\ov{q}}) g_{k \ov{\ell}} + \chi^{\ov{q}k} \chi^{\ov{\ell} p} (\partial_{\ov{j}} \chi_{p\ov{q}})(\partial_i g_{k \ov{\ell}}) \\
& \mbox{} + \chi^{\ov{s}k} \chi^{\ov{\ell}r} (\partial_i \chi_{r\ov{s}}) (\partial_{\ov{j}} g_{k \ov{\ell}}) - \chi^{\ov{\ell} k} \partial_i \partial_{\ov{j}} g_{k \ov{\ell}} \} \\
  = {} & \hat{g}^{\ov{j}i} h^{\ov{q}p} \partial_i \partial_{\ov{j}} \chi_{p \ov{q}} - \hat{g}^{\ov{j}i} h^{\ov{s} p} \chi^{\ov{q}r} (\partial_i \chi_{r\ov{s}})(\partial_{\ov{j}} \chi_{p\ov{q}}) - \hat{g}^{\ov{j}i} h^{\ov{q} r} \chi^{\ov{s}p} (\partial_i \chi_{r\ov{s}})(\partial_{\ov{j}} \chi_{p\ov{q}}) \\
& \mbox{} + 2\textrm{Re} \left( \hat{g}^{\ov{j}i} \chi^{\ov{q}k} \chi^{\ov{\ell}p} (\partial_{\ov{j}} \chi_{p\ov{q}} )(\partial_i g_{k \ov{\ell}}) \right) - \hat{g}^{\ov{j}i} \chi^{\ov{\ell}k} \partial_i \partial_{\ov{j}} g_{k \ov{\ell}}.
\end{align*}
And
\begin{align*}
\tilde{\Delta}\log u &=  \frac{\tilde{\Delta}u}{u} - \frac{ | \partial u|^2_h}{u^2} =\frac{1}{u} \left(  h^{\ov{\ell}k } \partial_k \partial_{\ov{\ell}} (\hat{g}^{\ov{j}i} \chi_{i \ov{j}})  - \frac{ | \partial u|^2_h}{u} \right) \\
&= \frac{1}{u} \left( h^{\ov{\ell}k} \hat{R}_{k \ov{\ell}}^{\ \ \ \ov{j}i} \chi_{i \ov{j}} + h^{\ov{\ell}k }\hat{g}^{\ov{j}i} \partial_k \partial_{\ov{\ell}} \chi_{i \ov{j}}- \frac{ | \partial u|^2_h}{u} \right),
\end{align*}
where $|\partial u|^2_h := h^{\ov{j}i} \partial_i u \partial_{\ov{j}}u$.
Then (\ref{evolve1}) follows from these two equations, the K\"ahler condition for $\chi$ and the fact that in  our coordinate system  we have
$$\hat{g}^{\ov{j}i} \chi^{\ov{\ell}k} \partial_i \partial_{\ov{j}} g_{k \ov{\ell}} =  \hat{g}^{\ov{j}i} \chi^{\ov{\ell} k} \hat{\nabla}_i \hat{\nabla}_{\ov{j}} g_{k \ov{\ell}} - \hat{g}^{\ov{j}i} \chi^{\ov{\ell}k} \hat{R}_{i \ov{j} \ \, \, \ov{\ell}}^{\ \ \, \ov{q}} g_{k \ov{q}}=   \hat{g}^{\ov{j}i}\chi^{\ov{\ell} k} \hat{\nabla}_i \hat{\nabla}_{\ov{j}} g_{k \ov{\ell}} - \chi^{\ov{\ell}k} \hat{R}_{\ \, \ov{\ell}}^{\ov{q}} g_{k \ov{q}}.$$

To deal with the terms involving one derivative of $\chi$ we use a completing the square argument, which is formally similar to that of Cherrier \cite{Chr} (see also \cite[Proposition 3.1]{TW}).
Compute 
$$K = \hat{g}^{\ov{\ell}i} \chi^{\ov{j}p} h^{\ov{q}k} B_{i \ov{j}k} \ov{B_{\ell \ov{p} q}} \ge 0,$$
where 
$$B_{i \ov{j}k} = \hat{\nabla}_i \chi_{k \ov{j}} - \chi_{i \ov{j}} \frac{\partial_k u}{u} - g^{\ov{b}a} \chi_{k \ov{b}} \hat{\nabla}_i g_{a\ov{j}}.$$
We compute
\begin{align*}
K  = {} & \hat{g}^{\ov{\ell}i} \chi^{\ov{j}p} h^{\ov{q}k} \hat{\nabla}_i \chi_{k \ov{j}} \hat{\nabla}_{\ov{\ell}} \chi_{p\ov{q}} + \hat{g}^{\ov{\ell}i} \chi^{\ov{j}p} h^{\ov{q}k} \chi_{i\ov{j}} \left( \frac{\partial_k u}{u} \right) \chi_{p\ov{\ell}} \left( \frac{\partial_{\ov{q}}u}{u} \right) \\
&  + \hat{g}^{\ov{\ell}i} \chi^{\ov{j}p} h^{\ov{q}k} g^{\ov{b}a} \chi_{k \ov{b}} (\hat{\nabla}_i g_{a\ov{j}}) g^{\ov{s}r} \chi_{r \ov{q}} \hat{\nabla}_{\ov{\ell}} g_{p \ov{s}} - 2 \textrm{Re} \left( \hat{g}^{\ov{\ell}i} \chi^{\ov{j} p} h^{\ov{q}k} (\hat{\nabla}_i \chi_{k \ov{j}}) \chi_{p\ov{\ell}} \left( \frac{\partial_{\ov{q}}u}{u} \right) \right) \\
& - 2 \textrm{Re} \left( \hat{g}^{\ov{\ell} i} \chi^{\ov{j}p} h^{\ov{q}k} (\hat{\nabla}_i \chi_{k \ov{j}}) g^{\ov{b}a} \chi_{a\ov{q}} \hat{\nabla}_{\ov{\ell}} g_{p \ov{b}}\right) + 2\textrm{Re} \left( \hat{g}^{\ov{\ell}i} \chi^{\ov{j}p} h^{\ov{q}k} \chi_{i \ov{j}} \left( \frac{\partial_k u}{u} \right) g^{\ov{s}{r}} \chi_{r \ov{q}} \hat{\nabla}_{\ov{\ell}} g_{p \ov{s}} \right).
\end{align*}
Using the definition of $h^{\ov{j}i}$ and the K\"ahler condition for $\chi$,
\begin{align*}
K  = {} & \hat{g}^{\ov{\ell}i} \chi^{\ov{j}p} h^{\ov{q}k} \hat{\nabla}_i \chi_{k \ov{j}} \hat{\nabla}_{\ov{\ell}} \chi_{p\ov{q}} + \frac{ | \partial u|^2_h}{u}  
 + \hat{g}^{\ov{\ell}i} \chi^{\ov{j}p} \chi^{\ov{q}c} \chi^{\ov{d}k} g_{c \ov{d}} g^{\ov{b}a} \chi_{k \ov{b}} g^{\ov{s}r} \chi_{r \ov{q}} (\hat{\nabla}_i g_{a\ov{j}})(\hat{\nabla}_{\ov{\ell}} g_{p \ov{s}}) \\
& - 2\textrm{Re} \left( \partial_k \left( \hat{g}^{\ov{\ell} i} \chi_{i \ov{\ell}} \right) h^{\ov{q}k} \frac{\partial_{\ov{q}}u}{u} \right) - 2 \textrm{Re} \left( \hat{g}^{\ov{\ell}i} \chi^{\ov{j}p} \chi^{\ov{q} c} \chi^{\ov{d}k} g_{c \ov{d}} (\hat{\nabla}_i \chi_{k \ov{j}}) g^{\ov{b}a} \chi_{a\ov{q}} (\hat{\nabla}_{\ov{\ell}} g_{p \ov{b}}) \right) \\
& + 2 \textrm{Re} \left( \hat{g}^{\ov{\ell}i} \chi^{\ov{j}p} \chi^{\ov{q} c} \chi^{\ov{d}k} g_{c\ov{d}} \chi_{i \ov{j}} \left( \frac{\partial_k u}{u} \right) g^{\ov{s}r} \chi_{r \ov{q}} (\hat{\nabla}_{\ov{\ell}} g_{p \ov{s}}) \right)\\
 = {} &  \hat{g}^{\ov{\ell}i} \chi^{\ov{j}p} h^{\ov{q}k} \hat{\nabla}_i \chi_{k \ov{j}} \hat{\nabla}_{\ov{\ell}} \chi_{p\ov{q}} - \frac{ | \partial u|^2_h}{u} + \hat{g}^{\ov{\ell}i} \chi^{\ov{j} p} g^{\ov{s}a} (\hat{\nabla}_i g_{a\ov{j}})(\hat{\nabla}_{\ov{\ell}}g_{p\ov{s}}) \\
&- 2 \textrm{Re} \left( \hat{g}^{\ov{\ell}i} \chi^{\ov{j}p} \chi^{\ov{b}k} (\hat{\nabla}_i \chi_{k \ov{j}})(\hat{\nabla}_{\ov{\ell}} g_{p\ov{b}}) \right) + 2 \textrm{Re} \left( \chi^{\ov{s} k}  \hat{g}^{\ov{\ell}p} \left( \frac{\partial_k u}{u} \right)  (\hat{\nabla}_{\ov{\ell}} g_{p \ov{s}}) \right).
\end{align*}
Combining this with (\ref{evolve1}) gives,
\begin{align} \nonumber
\left( \ddt{} - \tilde{\Delta} \right) \log u  = {}  & \frac{1}{u} \left\{- h^{\ov{\ell} k } \hat{R}_{k \ov{\ell}}^{\ \ \ \ov{j} i} \chi_{i \ov{j}} - \hat{g}^{\ov{j}i} h^{\ov{s}p} \chi^{\ov{q} r} \hat{\nabla}_i \chi_{r \ov{s}}\hat{\nabla}_{\ov{j}} \chi_{p \ov{q}} 
 - K \right. \\ \nonumber
 & \mbox{} + \hat{g}^{\ov{\ell}i} \chi^{\ov{j} p} g^{\ov{s}a} (\hat{\nabla}_i g_{a\ov{j}})(\hat{\nabla}_{\ov{\ell}}g_{p\ov{s}}) 
  + 2 \textrm{Re} \left(  \chi^{\ov{s} k}  \left( \frac{\partial_k u}{u} \right) \hat{g}^{\ov{\ell}p} (\hat{\nabla}_{\ov{s}} g_{p \ov{\ell}}) \right)\\ \nonumber
  & \mbox{} \left. - \hat{g}^{\ov{j}i} \chi^{\ov{\ell} k} \hat{\nabla}_i \hat{\nabla}_{\ov{j}} g_{k \ov{\ell}}  +  \chi^{\ov{\ell}k} \hat{R}_{\ \, \ov{\ell}}^{\ov{q}} g_{k \ov{q}}  \right\} \\ \nonumber
 \le {} & C \tr{h}{\hat{g}} + \frac{C}{u} (\tr{\chi}{\hat{\omega}})(\tr{\omega}{\hat{\omega}}) + 2 \textrm{Re} \left( \chi^{\ov{s} k}  \left( \frac{\partial_k u}{u^2} \right)  \partial_{\ov{s}} \tr{\hat{\omega}} \omega \right).\end{align}
Indeed, to see the last inequality, we estimate 
$$\frac{1}{u} | h^{\ov{\ell} k } \hat{R}_{k \ov{\ell}}^{\ \ \ \ov{j} i} \chi_{i \ov{j}}| \le C  \tr{h}{\hat{g}},$$
and
$$\frac{1}{u} \left|\hat{g}^{\ov{j}i} \chi^{\ov{\ell} k} \hat{\nabla}_i \hat{\nabla}_{\ov{j}} g_{k \ov{\ell}} \right| + \frac{1}{u}\left|  \chi^{\ov{\ell}k} \hat{R}_{\ \, \ov{\ell}}^{\ov{q}} g_{k \ov{q}} \right| + \frac{1}{u} \hat{g}^{\ov{\ell}i} \chi^{\ov{j} p} g^{\ov{s}a} (\hat{\nabla}_i g_{a\ov{j}})(\hat{\nabla}_{\ov{\ell}}g_{p\ov{s}})  \le \frac{C}{u} (\tr{\chi}{\hat{\omega}})(\tr{\omega}{\hat{\omega}}),$$
for a constant $C$ depending only on $\hat{g}$ and $\| g_0 \|_{C^2(M, \hat{g})}$.  Note that $\tr{g}{\hat{g}}$ is uniformly bounded from below away from zero, but may blow up along $D$.
This completes the proof of the lemma.

\end{proof}

Next we prove the estimate on $\chi$.

\begin{lemma}  There exist uniform constants $C$, $\gamma$, independent of $\ve$, such that
$$u=\emph{tr}_{\hat{\omega}}{\chi} \le \frac{C}{|s|_H^{2\gamma}}.$$
\end{lemma}
\begin{proof}    From the condition (\ref{alpha}), we may choose uniform positive constants $\eta$, $\delta$ and $\sigma$ to be sufficiently small so that
\begin{equation} \label{con1}
c \chi_0 - \omega - c\delta R_H > 3 \eta \omega,
\end{equation}
and
\begin{equation} \label{con2}
 \chi_0 - \delta R_H  \ge \sigma \hat{\omega},
\end{equation}
where we write $c= c_{\ve} = 2 [\chi_0]\cdot [\omega_{\ve}]/[\chi_0]^2$.

Define (cf. \cite{Ts})
$$\tilde{\varphi} = \varphi - \delta \log |s|^2_H.$$  Note that $\tilde{\varphi}$ is bounded from below, and tends to infinity along $D$.  Let $A>1$ be a large constant to be determined later.
Consider the evolution of the quantity 
$$Q = \log u - A \tilde{\varphi} + \frac{1}{\tilde{\varphi}+C_0},$$
where we choose the uniform constant  $C_0$  so that
$$0 \le \frac{1}{\tilde{\varphi}+C_0} \le 1.$$
This type of quantity was used by Phong-Sturm  \cite{PS} in their study of the degenerate complex Monge-Amp\`ere equation (for its later use in a parabolic setting, see \cite{TW}). 
Note that $Q$ achieves a maximum at each time $t$ away from $D$.   Since $\varphi$ is uniformly bounded, it suffices to bound $Q$ from above at its maximum, as long as $A$ is chosen uniformly.
 
Note that at a maximum of $Q$ we have
$$\frac{\partial_k u}{u} = \left(A + \frac{1}{(\tilde{\varphi} +C_0)^2}\right) \partial_k \tilde{\varphi}.$$
Then at a maximum of $Q$ we have, from Lemma \ref{lemmaevolve},
\begin{align} \nonumber
\left( \ddt{} - \tilde{\Delta} \right)Q  \le {}  &C_1\tr{h}{\hat{g}} +  \frac{C (\tr{\chi}{\hat{\omega}}) (\tr{\omega}{\hat{\omega}})}{u}  + \frac{2}{u} \textrm{Re} \left( \chi^{\ov{s} k} \left( A+ \frac{1}{(\tilde{\varphi} + C_0)^2} \right) (\partial_k \tilde{\varphi}) \partial_{\ov{s}} \tr{\hat{\omega}}{\omega} \right) \\ \label{evQ1}
& - \left( A + \frac{1}{(\tilde{\varphi}+C_0)^2} \right) \left( \ddt{} - \tilde{\Delta} \right) \tilde{\varphi} - \frac{2}{(\tilde{\varphi} + C_0)^3} | \partial \tilde{\varphi}|^2_h. 
\end{align}
But from (\ref{omega0}) and (\ref{lbchi}),
\begin{equation} \label{term2}
 \frac{C (\tr{\chi}{\hat{\omega}}) (\tr{\omega}{\hat{\omega}})}{u} \le \frac{C'(\tr{\chi}{\omega})(\tr{\hat{\omega}}{\hat{\omega}})}{u|s|_H^{4\beta}} \le
 \frac{C''}{u |s|_H^{4\beta}}.
 \end{equation}
Observe that at a maximum of $Q$ we may assume that  $\frac{C''}{u |s|_H^{4 \beta}} \le 1$.  Indeed if not, then assuming that $\delta A \ge 2\beta$,
we have that $\log u + A \delta \log |s|^2_H \le C$ and it follows immediately that  $Q$ is bounded from above, which is what we need to show.  

Next,
\begin{align*}
\frac{2}{u} \textrm{Re} \left(\chi^{\ov{s} k} \left( A+ \frac{1}{(\tilde{\varphi} + C_0)^2} \right) (\partial_k \tilde{\varphi}) \partial_{\ov{s}} \tr{\hat{\omega}}{\omega} \right) &\le \frac{CA}{u} | \partial \tilde{\varphi}|_{\chi} \, | \partial \tr{\hat{\omega}}{\omega} |_{\chi}.
\end{align*}
But
$$\frac{1}{u} | \partial \tilde{\varphi} |^2_{\chi} = \frac{1}{u} \chi^{\ov{j}i} \partial_i \tilde{\varphi} \partial_{\ov{j}} \tilde{\varphi} \le \chi^{\ov{j}k} \chi^{\ov{\ell}i} \hat{g}_{k \ov{\ell}} \partial_i \tilde{\varphi} \partial_{\ov{j}} \tilde{\varphi}\le \frac{C}{|s|_H^{2\beta}} h^{\ov{j}i} \partial_i \tilde{\varphi} \partial_{\ov{j}} \tilde{\varphi} = \frac{C}{|s|_H^{2\beta}} | \partial \tilde{\varphi}|^2_h,$$
and, using (\ref{lbchi}),
$$| \partial \tr{\hat{\omega}}{\omega} |_{\chi}^2 \le \frac{C}{|s|^{2\beta}_H}.$$
Hence
\begin{align*}
\frac{CA}{u} | \partial \tilde{\varphi}|_{\chi} \, | \partial \tr{\hat{\omega}}{\omega} |_{\chi} & \le \frac{C'A}{ \sqrt{u} |s|^{2\beta}_H} | \partial \tilde{\varphi}|_h \\
 & \le 2 \frac{ | \partial \tilde{\varphi}|^2_h}{(\tilde{\varphi} +C_0)^3} + \frac{C'' A^2 (\tilde{\varphi}+C_0)^3}{u |s|_H^{4\beta}}.
\end{align*}


Putting this together, we obtain
\begin{align*}
\frac{2}{u} \textrm{Re} \left(\chi^{\ov{s} k} \left( A+ \frac{1}{(\tilde{\varphi} + C_0)^2} \right) (\partial_k \tilde{\varphi}) \partial_{\ov{s}} \tr{\hat{\omega}}{\omega} \right) &\le 2 \frac{ | \partial \tilde{\varphi}|^2_h}{(\tilde{\varphi} +C_0)^3}  +1,
\end{align*}
since we may assume without loss of generality, by an argument similar to the one given above, that $C''A^2(\tilde{\varphi}+C_0)^3/(u|s|^{4\beta}_H) \le 1$.
Combining this with (\ref{evQ1}) and (\ref{term2}), we obtain at a maximum point of $Q$,
\begin{align} \label{Q1}
0 \le \left( \ddt{} - \tilde{\Delta} \right)Q & \le  C_1\tr{h}{\hat{g}} +2 - \left( A + \frac{1}{(\tilde{\varphi}+C_0)^2} \right) \left( \ddt{} - \tilde{\Delta} \right) \tilde{\varphi}.
\end{align}
Now compute on $M \setminus D$,
\begin{align} \nonumber
\left( \ddt{} - \tilde{\Delta} \right) \tilde{\varphi}   = {}  & c - \tr{\chi}{\omega} - h^{\ov{j}i} \partial_i \partial_{\ov{j}} (\varphi - \delta \log |s|^2_H) \\ \nonumber
  ={} &  c - 2 \chi^{\ov{j}i} g_{i \ov{j}} + h^{\ov{j}i} ((\chi_0)_{i \ov{j}} - \delta (R_H)_{i \ov{j}}) \\ \nonumber
 = {} & c - 2 \chi^{\ov{j}i} g_{i \ov{j}} + \eta h^{\ov{j}i} ( (\chi_0)_{i \ov{j}} - \delta (R_H)_{i \ov{j}}) \\ \label{varphi1}
& \mbox{} + \frac{(1-\eta)}{c}  h^{\ov{j}i} (c\, (\chi_0)_{i \ov{j}} - c\delta (R_H)_{i \ov{j}}) ,
\end{align}
for $\eta>0$ as in (\ref{con1}).  From (\ref{con2}), we choose $A$ (depending only on $C_1$, $\eta$ and $\sigma$) sufficiently large so that
$$A \eta h^{\ov{j}i} ((\chi_0)_{i \ov{j}} - \delta (R_H)_{i \ov{j}}) \ge  C_1 \tr{h}{\hat{g}}.$$
Then from this together with (\ref{Q1}) and (\ref{varphi1}), we obtain
\begin{align}
c - 2 \chi^{\ov{j}i} g_{i \ov{j}} + \frac{(1-\eta)}{c}  h^{\ov{j}i} (c\, (\chi_0)_{i \ov{j}} - c\delta (R_H)_{i \ov{j}}) \le \frac{2}{A},
\end{align}
and hence from (\ref{con1}), 
\begin{align}
c - 2 \chi^{\ov{j}i} g_{i \ov{j}} + \frac{(1-\eta)(1+3\eta)}{c}  h^{\ov{j}i} g_{i\ov{j}} \le \frac{2}{A}.
\end{align}
This implies that, shrinking $\eta$ if necessary,
$$c - 2 \chi^{\ov{j}i} g_{i \ov{j}} + \frac{(1+ \eta)}{c}  h^{\ov{j}i} g_{i\ov{j}} \le \frac{2}{A}.$$
Now choosing coordinates for which $g$ is the identity and  $\chi$ is diagonal with entries $\lambda_1, \lambda_2$, we have
$$c + \frac{(1+\eta)}{c} \sum_{i=1}^2 \frac{1}{\lambda_i^2} - 2 \sum_{i=1}^2 \frac{1}{\lambda_i} \le \frac{2}{A}.$$
Completing the square as in \cite{W2, SW1}, we get
$$\sum_{i=1}^2 \left( \frac{\sqrt{c}}{\sqrt{1+\eta}} - \frac{\sqrt{1+\eta}}{\sqrt{c} \, \lambda_i} \right)^2 \le \frac{2}{A} - c + \frac{2c}{1+\eta}= \frac{2}{A} + \frac{c(1-\eta)}{1+\eta}.$$
We may assume  that $A$ is chosen large enough so that 
$$\frac{2}{A} \le \eta\frac{c(1-\eta)}{1+\eta},$$
and thus
$$\sum_{i=1}^2 \left( \frac{\sqrt{c}}{\sqrt{1+\eta}} - \frac{\sqrt{1+\eta}}{\sqrt{c} \, \lambda_i} \right)^2 \le c(1-\eta).$$
Hence, for $i=1,2$,
$$\frac{\sqrt{c}}{\sqrt{1+\eta}} - \frac{\sqrt{1+\eta}}{\sqrt{c} \, \lambda_i} \le \sqrt{c(1-\eta)},$$
which implies that
$$\lambda_i \le \frac{1+\eta}{c(1-\sqrt{1-\eta^2})}.$$
Then  $$\tr{g}{\chi} \le C,$$
at this maximum point of $Q$.  Hence $\tr{\hat{g}}{\chi} \le C$ at this point, and we see that $Q$ is bounded from above.
This completes the proof.
\end{proof}

The higher order estimates (\ref{ckest}) follows immediately by applying the standard local parabolic theory (as in \cite{SW1}, for example).  This completes the proof of Proposition \ref{mainprop}.

\section{Proof of the main theorem and corollaries} \label{sectionproofs}

We can now prove the main results of the paper. 

\setcounter{equation}{0}

\medskip

\noindent{\it Proof of Theorem \ref{maintheorem}. }   From Proposition \ref{mainprop}, we can find a sequence $\ve_j \rightarrow 0$ such that $\varphi_{\ve_j}$ converges in $C^{\infty}$ on compact subsets of $(M \setminus D) \times [0,\infty)$.  Define on $M \setminus D$,
$$\varphi= \lim_{j \rightarrow \infty} \varphi_{\ve_j}.$$
Then on $(M \setminus D) \times [0,\infty)$,  $\varphi$ is smooth, satisfies $\chi_0 + dd^c \varphi > 0$ and solves the degenerate J-flow equation (\ref{jflowdegen}).
  Moreover, again from Proposition \ref{mainprop},  $\sup_{M\setminus D} |\varphi|$ and $\sup_{M\setminus D}|\dot{\varphi}|$ are uniformly bounded independent of $t$.  It is a standard result in pluripotential theory that a smooth function $\varphi$ on $M-D$ which satisfies $\chi_0 + dd^c \varphi > 0$ and $\sup_{M-D} |\varphi| \le C$  can be extended uniquely  to an element of $\mathcal{P}_{\chi_0}^{\textrm{weak}}$ (see \cite{K3} for example).   Writing again $\varphi$ for this function, we obtain the required solution $\varphi$ to (\ref{jflowdegen}).
  
Now recall that the $\mathcal{J}$-functional is defined on $\mathcal{P}_{\chi_0}$ by (\ref{Jfunctional}).  One can also write down an explicit formula:
\begin{equation} \label{Jfunctional2}
\mathcal{J}_{\omega_0, \chi_0}(\varphi) = \int_M \varphi (\chi_{\varphi} \wedge \omega_0 + \chi_0 \wedge \omega_0) - \frac{c_0}{3} \int_M  \varphi ( \chi_{\varphi}^2 + \chi_{\varphi} \wedge \chi_0 + \chi_0^2), \quad \varphi \in \mathcal{P}_{\chi_0},
\end{equation}
and this definition extends to $\varphi \in \mathcal{P}_{\chi_0}^{\textrm{weak}}$.   From the uniform $L^{\infty}$ bound for the solution $\varphi(t)$ of the degenerate J-flow, as argued in \cite{FLSW}, we see that
\begin{equation}
\mathcal{J}_{\omega_0, \chi_0}(\varphi(t)) \ge -C, 
\end{equation}
for a uniform constant $C$ independent of $t$.  In addition, $\mathcal{J}_{\omega_0, \chi_0}(\varphi(t))$ satisfies
\begin{equation}
\frac{d}{dt} \mathcal{J}_{\omega_0, \chi_0}(\varphi(t)) = - \int_{M\setminus D} \dot{\varphi}^2 \chi_{\varphi(t)}^2 \le 0.
\end{equation}
Then it follows from the proof of Theorem 1.1 in \cite{FLSW} that $\dot{\varphi}$ tends to zero in $C^{\infty}$ on compact subsets of $M \setminus D$.   

Next, define the $\mathcal{I}$-functional on $\mathcal{P}^{\textrm{weak}}_{\chi_0}$ by
$$\mathcal{I}_{\omega_0, \chi_0}(\varphi) = \frac{1}{3} \int_M  \varphi ( \chi_{\varphi}^2 + \chi_{\varphi} \wedge \chi_0 + \chi_0^2),$$
and we see that $\mathcal{I}_{\omega_0, \chi_0}(\varphi(t)) = \mathcal{I}_{\omega_0, \chi_0}(\varphi_0)$ for all $t$. It follows from the same argument as in \cite{FLSW} that as
$t \rightarrow \infty$, the solution $\varphi(t)$ to the degenerate J-flow converges in $C^{\infty}$ on compact subsets of $M\setminus D$ to the unique $\varphi_{\infty} \in \mathcal{P}_{\chi_0}^{\textrm{weak}}$ satisfying 
 the critical equation
 $$2 \chi_{\varphi_{\infty}} \wedge \omega_0 = c_0 \chi^2_{\varphi_{\infty}}$$
 on $M \setminus D$ subject to the normalization condition $\mathcal{I}_{\omega_0, \chi_0} (\varphi_{\infty}) = \mathcal{I}_{\omega_0, \chi_0} (\varphi_0).$   Indeed, to see this last uniqueness statement, observe that $\varphi_{\infty}$ must satisfy
 \begin{equation} \label{cma}
(\alpha_0 + c_0 dd^c \varphi_{\infty})^2 = \omega_0^2, \quad \alpha_0 +c_0 dd^c \varphi_{\infty} >0 \ \textrm{on } M \setminus D,
\end{equation} 
for $\alpha_0 = c_0 \chi_0 - \omega_0>0$.  Moreover, $c_0 \varphi_{\infty}$ lies in $\mathcal{P}_{\alpha_0}^{\textrm{weak}}$.
But such  solutions of the complex Monge-Amp\`ere equation (\ref{cma})  are unique up to the addition of a constant \cite[Corollary 4.2]{K2}.

It remains to prove the uniqueness of the solution $\varphi(t)$ to the degenerate J-flow.  We use an argument similar to one given in \cite{ST}.
Suppose there is another solution  $\psi(t) \in \mathcal{P}_{\chi_0}^{\textrm{weak}}$ of (\ref{jflowdegen})  satisfying $\sup_{M \setminus D} | \dot{\psi}| \le C$.
Define $\theta_\delta= \varphi - \psi - \delta \log |s|_H^2$ on $M \setminus D$, which tends to infinity along $D$. 
For $v\in [0,1]$, let $\eta_v = v \chi_\varphi + (1-v) \chi_\psi$, $\tau_v^{\bar\ell  k } = \eta_v^{\bar j k} \eta_v^{\bar \ell i} (g_0)_{i\bar j}$. Computing as in (\ref{thetaevolve}), we have on $M \setminus D$,
$$\ddt{\theta_\delta} =  \left( \int_0^1 \tau_v^{\ov{\ell}k} dv \right) \partial_k \partial_{\ov{\ell}}  \theta_\delta - \delta \left( \int_0^1 \tau_v^{\ov{\ell}k} dv \right) (R_H)_{k\bar \ell}.$$
Fix a time interval $[0, T]$.  From the estimates $\sup_{M \setminus D} | \dot{\varphi}| \le C$ and $\sup_{M \setminus D} | \dot{\psi}| \le C$ we have the estimate $\eta_v \ge \frac{1}{C} \omega_0$ for a uniform constant $C>0$.  It follows that $(\tau_v^{\ov{\ell}k}) \le C (g_0^{\ov{\ell}k})$.  Then
\begin{align*}
\delta \left( \int_0^1 \tau_v^{\ov{\ell}k} dv \right) (R_H)_{k\bar \ell} & \le C \delta g_0^{\ov{\ell}k} (R_H)_{k \ov{\ell}} \le \frac{2C \delta}{\rho}.
\end{align*}
since, from (\ref{omega0}) we have $\omega_0 - \rho R_H >0$ for a uniform $\rho>0$.

Hence
$$\ddt{\theta_\delta} \ge   \left( \int_0^1 \tau_s^{\ov{\ell}k} ds \right) \partial_k \partial_{\ov{\ell}}  \theta_\delta - \frac{2C \delta}{\rho},$$
and so by the maximum principle, we have 
$$\theta_\delta \geq -A\delta t \geq - A\delta T,$$
for a uniform constant $A$. 
 It follows that $\varphi \geq \psi + \delta \log |s|^2_H - A\delta T$ and so $\varphi\geq \psi$ after letting $\delta \rightarrow 0$. The same argument shows that $\psi \geq \varphi$ and so $\varphi = \psi.$ \qed

\bigskip 
\noindent {\it Proof of Corollary \ref{corollarylb}.}  \ Given the discussion above, this is now immediate, since for any $\varphi_0 \in \mathcal{P}_{\chi_0}$, we have $\mathcal{J}_{\omega_0, \chi_0} (\varphi_0) \ge \lim_{t\rightarrow\infty} \mathcal{J}_{\omega_0,\chi_0}(\varphi(t)) = \mathcal{J}_{\omega_0, \chi_0} (\varphi_{\infty})$.  For the last equality, we have used Lemma 3.2 in \cite{FLSW}.  \qed

\bigskip

\noindent {\it Proof of Corollary \ref{mabuchithm}.} \ The Mabuchi energy on the K\"ahler class $[\chi_0]$ is defined by 
$$\mathcal{M}_{\chi_0} (\varphi)= - \int_0^1 \int_M \dot{\varphi}_s (R_{\chi_{\varphi_s}} - \underline{R}) \chi_{\varphi_s}^2 ds, $$ where $\varphi_s$ is a smooth path in $\mathcal{P}_{\chi_0}$ between $0$ and $\varphi$, and $\underline{R}$ is  the average scalar curvature $\underline{R}= \frac{1}{\int_M \chi_0^2} \int_M R_{\chi_0} \chi_0^2.$  Let
$$E_{\chi_0}(\varphi) = \sqrt{-1}  \int_M \partial \varphi \wedge \dbar \varphi \wedge (\chi_0 +   \chi_{\varphi}).$$ 
be the well-known Aubin-Yau functional (often denoted by $I_{\chi_0}$). Then we say the Mabuchi energy is \emph{proper} \cite{T2} if there exists an increasing  function $f: [0, \infty) \rightarrow \mathbb{R}$, satisfying $\lim_{x\rightarrow \infty} f(x) = \infty$,  such that for all $\varphi \in \mathcal{P}_{\chi_0}$, 
$$\mathcal{M}_{\chi_0}(\varphi) \geq f(E_{\chi_0}(\varphi))).$$
In fact, for the purposes of this corollary, we may take $f$ to be linear (cf. \cite{PSSW}).

Since $K_M$ is big and nef,  it is well-known that there exists a closed nonnegative $(1,1)$ form $\omega_0 \in c_1(K_M)$ satisfying (\ref{omega0}).  Indeed, one can take a Fubini-Study metric induced from the pluricanonical system $|mK_M|$ for sufficiently large $m$, and divide by $m$ to obtain  a smooth  closed nonnegative $(1,1)$ form $\omega_0 \in c_1(K_M)$.  Note that since $\omega_0$ is the pull-back of a holomorphic (and hence smooth) map from $M$ into projective space, it is smooth everywhere on $M$.  However, it is only positive definite on $M \setminus D$ where $D$ is the base locus.  Moreover, $[\omega_0] - \rho c_1([D])$ is K\"ahler for all $\rho>0$ sufficiently small.  Hence we can find a Hermitian metric $H$ on $[D]$ so that $\omega_0 - \rho R_H  \ge \frac{1}{C_0} \hat{\omega}$ for some fixed K\"ahler metric $\hat{\omega}$ and a positive constants $C_0, \rho$.  By definition of $\omega_0$, we have $\omega_0 \ge \frac{1}{C_0}|s|_H^{2\beta} \hat{\omega}$, for some positive $\beta$, after possibly increasing $C_0$.

  The condition (\ref{assume2}) implies that
\begin{equation} \label{cohom}
 c_0  [\chi_0] - [\omega_0]>0, \quad \textrm{for} \quad c_0= \frac{[\chi_0] \cdot [\omega_0] }{[\chi_0]^2}.
\end{equation}
Hence we can apply Corollary \ref{corollarylb} to see that $\mathcal{J}_{\omega_0, \chi_0}$ is uniformly bounded from below.  The formula of Chen \cite{Ch1} gives
$$\mathcal{M}_{\chi_0} = \mathcal{J}_{\omega_0, \chi_0} + \mathcal{F},$$
for a certain functional $\mathcal{F}$, which is proper on $\mathcal{P}_{\chi_0}$ \cite{SW1, T2}.  This completes the proof. \qed

\begin{remark}
\emph{
We remark that one can give an alternative proof of these two corollaries by elliptic methods.  However, we believe that the degenerate J-flow is interesting in its own right, and may be important in extending these results  to higher dimensions.}
\end{remark}


\begin{thebibliography}{}

\bibitem[B]{B} B{\l}ocki, Z. {\em Weak solutions to the complex Hessian equation}, Ann. Inst. Fourier (Grenoble) 55 (2005), no. 5, 1735--1756

\bibitem [Ch1] {Ch1} Chen, X., \emph{On the lower bound of the Mabuchi energy and its application}, Int. Math. Res. Notices 12 (2000), 607--623

\bibitem [Ch2] {Ch2} Chen, X., \emph{A new parabolic flow in K\"{a}hler manifolds}, Comm. Anal. Geom. 12, no.4 (2004), 837--852

\bibitem[Chr]{Chr} Cherrier, P., {\em  \'Equations de Monge-Amp\`ere sur les vari\'et\'es hermitiennes compactes}, Bull. Sci. Math. (2) 111 (1987), no. 4, 343--385

\bibitem[Chi]{Chi} Chinh, L.H. {\em Solutions to degenerate complex Hessian equations}, preprint, arXiv: 1202.2436

\bibitem[DK]{DK} Dinew, S. and Ko{\l}odziej, S., {\em Liouville and Calabi-Yau type theorems for complex Hessian equations}, preprint, arXiv.org/1203.3995

\bibitem [D1] {D1} Donaldson, S.K.,\emph{ Moment maps and diffeomorphisms}, Asian J. Math. 3, no.
1 (1999), 1--16

\bibitem[D2]{D2} Donaldson, S.K. \emph{Scalar curvature and stability of toric varieties}, J. Differential Geom. 62 (2002), 289Ð349

\bibitem [FL1] {FL1} Fang, H. and Lai, M., \emph{On the geometric flows solving K\"{a}hlerian inverse $\sigma_k$ equations}, Pacific J. Math. 258 (2012), no. 2, 291--304

\bibitem [FL2] {FL2} Fang, H. and Lai, M., \emph{Convergence of general inverse $\sigma_k$-flow on K\"{a}hler manifolds with Calabi Ansatz}, preprint, arxiv:1203.5253

\bibitem [FLM] {FLM} Fang, H., Lai, M. and Ma, X., \emph{On a class of fully nonlinear flows in K\"{a}hler geomety}, J. Reine Angew. Math. 653 (2011), 189--220

\bibitem[FLSW]{FLSW} Fang, H., Lai, M., Song, J. and Weinkove, B., \emph{ The $J$-flow on K\"{a}hler surfaces: a boundary case}, preprint, arXiv:1204.4068 

\bibitem[GL]{GL} Guan, B. and Li, Q. {\em A Monge-Ampere Type Fully Nonlinear Equation on Hermitian Manifolds},
Discrete and continuous dynamical systems, series B, 17 (2012), no. 6, 1991--1999

\bibitem[H]{H} Hou, Z. {\em Complex Hessian equation on K\"ahler manifold}, Int. Math. Res. Not. IMRN 2009, no. 16, 3098--3111

\bibitem[HMW]{HMW} Hou, Z., Ma, X.N. and Wu, D, {\em A second order estimate for complex Hessian equations on a compact K\"ahler manifold}, Math. Res. Lett. 17 (2010), no. 3, 547--561

\bibitem[K1]{K1} Ko{\l}odziej, S. {\em The complex Monge-Amp\`ere equation}, Acta Math.  180 (1998), no. 1, 69--117

\bibitem[K2]{K2} Ko{\l}odziej, S. {\em The Monge-Amp\`ere equation on compact K\"ahler manifolds}, Indiana Univ. Math. J. 52 (2003), no. 3, 667--686

\bibitem[K3]{K3} Ko{\l}odziej, S. {\em The complex Monge-Amp\`ere equation and pluripotential theory}, Mem. Amer. Math. Soc. 178 (2005), no. 840, x+64 pp

\bibitem[LiS]{LiS} Li, S.-Y. {\em On the Dirichlet problems for symmetric function equations of the eigenvalues of the complex Hessian}, 
Asian J. Math. 8 (2004), no. 1, 87--106

\bibitem[LiY]{LiY} Li, Y. {\em A priori estimates on Donaldson equation over compact Hermitian manifolds}, preprint, arXiv: 1210.0254

\bibitem[PSSW]{PSSW} Phong, D.H., Song, J.,  Sturm, J. and Weinkove, B., \emph{The Moser-Trudinger inequality on K\"ahler-Einstein manifolds}, Amer. J. Math. 130 (2008), no. 4, 1067--1085

\bibitem[PS]{PS} Phong, D.H., Sturm, J. {\em  The Dirichlet problem for degenerate complex Monge-Amp\`ere
equations}, Comm. Anal. Geom. 18 (2010), no. 1, 145--170.

\bibitem[PR]{PR} Panov, D. and Ross, J. {\em Slope stability and exceptional divisors of high genus}, Math. Ann. 343 (2009), no. 1, 79--101

\bibitem[RT]{RT} Ross, J. and Thomas, R. {\em An obstruction to the existence of constant scalar curvature K\"ahler metrics}, J. Differential Geom. 72 (2006), no. 3, 429--466

\bibitem[ST]{ST} Song, J, and Tian, G. {\em  The K\"ahler-Ricci flow through singularities}, preprint, arXiv:0909.4898 [math.DG]

\bibitem [SW1] {SW1} Song, J. and Weinkove, B., \emph{ The convergence and singularities of the J-flow with applications to the Mabuchi energy},
Comm. Pure Appl. Math. 61 (2008), no. 2, 210--229

\bibitem[SW2]{SW2} Song, J. and Weinkove, B. \emph{Contracting exceptional divisors by the K\"ahler-Ricci flow}, Duke Math. J. 162 (2013), no. 2, 367--415

\bibitem[T1]{T1} Tian, G. {\em K\"ahler-Einstein metrics with positive scalar curvature}, Invent. Math. 130 (1997), no. 1, 1--37
\bibitem[T2]{T2} Tian, G. {\em Canonical metrics in K\"ahler geometry}. Notes taken by Meike Akveld. Lectures in Mathematics ETH Z\"urich. Birkh\"auser Verlag, Basel, 2000


\bibitem[TW]{TW} Tosatti, V. and Weinkove, B. {\em On the evolution of a Hermitian metric by its Chern-Ricci form}, preprint, arXiv: 1201.0312

\bibitem[Ts]{Ts} Tsuji, H. {\em Existence and degeneration of K\"ahler-Einstein metrics on minimal algebraic varieties of general type}, Math.
Ann. 281 (1988), 123--133

\bibitem [W1] {W1} Weinkove, B., \emph{Convergence of the $J$-flow on K\"{a}hler surfaces}, Comm. Anal. Geom. 12, no. 4 (2004), 949--965

\bibitem [W2] {W2} Weinkove, B., \emph{On the $J$-flow in higher dimensions and the lower boundedness of the Mabuchi energy}, J. Differential Geom.  73  (2006),  no. 2, 351--358

\bibitem [Y1] {Y1}  Yau, S.-T., \emph{ On the Ricci curvature of a compact K\"{a}hler manifold and the complex Monge-
Amp\`{e}re equation}, I, Comm. Pure Appl. Math. 31 (1978), 339--411
\bibitem[Y2]{Y2} Yau, S.-T., {\em Open problems in geometry}, Proc. Symposia Pure Math. 54 (1993), 1--28


\end{thebibliography}
\end{document}